\documentclass[12pt,a4paper]{article}
\usepackage[utf8]{inputenc}
\usepackage{amsmath}
\usepackage{amsfonts}
\usepackage{amssymb}
\usepackage{graphicx}
\usepackage{mathrsfs}
\usepackage{eqnarray,amsthm}
\usepackage{lmodern}
\usepackage[T1]{fontenc}
\usepackage{fullpage}

\begin{document}

\title{\textbf{Flux-limited solutions and state constraints for quasi-convex Hamilton-Jacobi equations}}

\author{Jessica Guerand\footnote{ D\'epartement de Mathématiques et applications, \'Ecole Normale Sup\'erieure (Paris), 
45 rue d’Ulm, 75005 Paris, France. \texttt{jessica.guerand@ens.fr}}}

\maketitle

\newtheorem{thm}{\bf Theorem}[section]
\newtheorem{prop}[thm]{Proposition}
\newtheorem{defin}[thm]{Definition}
\newtheorem{Lm}[thm]{Lemma}
\newtheorem{cor}[thm]{Corollary}
\newtheorem{theo}[thm]{Theorem}
\theoremstyle{remark}
\newtheorem{Rmk}[thm]{Remark}

\begin{center}
\textbf{Abstract}
\end{center}

Imbert and Monneau proved that Soner's formulation of state constraint problems on a bounded interval is equivalent to the so-called flux-limited formulation they introduced recently. In the multidimensional setting, we show this result for a general $\mathcal{C}^{1}$ bounded open set of $\mathbb{R}^{d}$ and obtain the result in the stationary and the evolution type cases. 
\bigbreak
\bigbreak

\section{Introduction}
\subsection{Hamilton-Jacobi equation and state constraint problems}

State constrainted Hamilton-Jacobi equations naturally appear in optimal control problems where trajectories must remain in a domain $\Omega$. Introducing flux-limited solutions, Imbert and Monneau study a more general boundary condition and prove existence and uniqueness results for multi-junction domain in dimension 1 in space \cite{im} and in the multidimensional case \cite{im2}.

Let us recall Soner's formulation of state constraint (SC) problems \cite{son}: a function $u\in C(\bar{\Omega})$ is a viscosity solution of the SC problem if it satisfies in the viscosity sense, 
\begin{equation}
\label{SC}
 \left\{\begin{array}{lll}
u+H(\nabla u)= 0 & \mbox { in } & \Omega\\
u+H(\nabla u)\geq 0 & \mbox { on } & \partial\Omega.
\end{array}\right.
\end{equation}
This formulation seems not to take into account the boundary behavior of $u$ as a sub-solution. However, Soner proves \cite{son} the uniqueness of bounded uniformly continuous solutions of (\ref{SC}) for a convex Hamiltonian $H$.  

In \cite{ish}, also for a convex Hamiltonian $H$, Ishii and Koike show the uniqueness of solutions not necessarily continuous for the same equation than (\ref{SC}) with an extra inequation for subsolutions,
\begin{equation}
\label{SC2}
 \left\{\begin{array}{lll}
u+H(\nabla u)= 0 & \mbox { in } & \Omega\\
u+H(\nabla u)\geq 0 & \mbox { on } & \partial\Omega\\
u+H_{in}(\nabla u)\leq 0 & \mbox { on } &  \partial\Omega,
\end{array}\right.
\end{equation}
 where $H_{in}$ is an «inward Hamiltonian» \cite{ish}.
 
In \cite{im}, for a quasi-convex continuous coercive Hamiltonian $H$ and $\Omega$ a bounded interval, Imbert and Monneau proved that the SC problem (\ref{SC}) is equivalent to the flux-limited problem,
\begin{equation}
\label{SC3}
 \left\{\begin{array}{lll}
u+H(\nabla u)= 0 & \mbox { in } & \Omega\\
u+H^{-}(\nabla u)\geq 0 & \mbox { on } & \partial\Omega\\
u+H^{-}(\nabla u)\leq 0 & \mbox { on } &  \partial\Omega,
\end{array}\right.
\end{equation} 
where $H^{-}$ is the nonincreasing part of the Hamiltonian along the inward vector defined in the next section. 
They also proved the result for the evolution type equation. 

In this paper, we prove the same result than Imbert and Monneau for a general $\mathcal{C}^{1}$ bounded open set $\Omega$ of $\mathbb{R}^{d}.$ The function $H^{-}$ is  the same as $H_{in}$ in \cite{ish} if $H$ and the domain $\Omega$ satisfy the previous assumptions.  

There are links between those three problems in some cases. We have the direct following implication. A solution of (\ref{SC3}) is also a solution of (\ref{SC2}) which is also a solution of (\ref{SC}). Imbert and Monneau show that a function is a solution of (\ref{SC}) if and only if it is a solution of (\ref{SC3}) with the previous assumptions on $H$ and for an open interval. We show it in the multidimensional setting. This result shows that the uniqueness obtained by Ishii and Koike for the problem (\ref{SC2}) implies the uniqueness for the problem (\ref{SC}) even for not necessarily continuous solutions. But with some restrictive assumptions: a $\mathcal{C}^{1}$ domain, a coercive Hamiltonian and a weak continuity condition has to be imposed to sub-solutions.
The main contribution of this article is two-fold: on the one hand, we can deal with equation of evolution type; on the other hand, the Hamiltonian is not necessarily convex.

The proof is essentially based on the ideas in \cite{im,im2} where the authors show that it is possible to reduce the class of test functions. This new class contains only test functions with one slope in space. To obtain this result for a multi-dimensional regular bounded open set, we introduce a local parametrization of the boundary. After the proof of the theorem, we provide direct proofs in particular cases. First we give a direct proof for the general super-solution case, second in dimension 1, we prove the result in the stationary case.  

There are other related works. For example, Hamilton-Jacobi equations posed on two domains separated by a frontier where some transmission conditions should be imposed, see \cite{RSZ,RZ}. In \cite{achtch,acct}, the authors study an optimal control problem where trajectories of the controlled system must remain in the embedded network.

\subsection{Main theorem}

In order to state our main theorem, let us first recall the notation of the article \cite{im} and \cite{im2}. Let T be positive.
Let $\Omega$ be a bounded open set of $\mathbb{R}^{d+1}$ with a boundary of class $\mathcal{C}^{1}$.

The Hamiltonian $H:\mathbb{R}^{d+1}\rightarrow \mathbb{R}$ is assumed to satisfy the following properties
\begin{equation}
\label{hyp}
\left\{\begin{array}{ll}
 \textbf{continuous,} \\
 \textbf{coercive}, & \lim_{|p|\rightarrow+\infty}H(p)=+\infty \\
\textbf{quasi-convex}, & \forall \lambda \in \mathbb{R},\quad \{ H\leq \lambda\} \mbox { is convex}.  
\end{array}\right.
\end{equation}

We define $H^{-}:\mathbb{R}^{d+1}\rightarrow \mathbb{R}$ (resp. $H^{+}$) the nonincreasing (resp. nondecreasing) part of the Hamiltonian $H$ by introducing few notations. 
Let $x\in \partial\Omega,$ $p\in \mathbb{R}^{d+1}$ and $n_{x}$ be the outward normal vector in $x$. The component of $p$ in the tangent plane to $\partial \Omega$ in $x$ is denoted by $p'$  and the component along $-n_{x}=N_{x}$ is denoted by  $p_{N}$,  so  $p_{N}=p\cdot N_{x}$ and $p=(p',p_{N})\in \mathbb{R}^{d}\times\mathbb{R}.$
By the same way, the tangential componant of  $\nabla u$ is denoted by $\nabla'u$ and let $\partial_{N}u=\nabla u \cdot N_{x}.$ Let $H_{x}:\mathbb{R}^{d+1}\rightarrow \mathbb{R}$ be the Hamiltonian defined with these components depending on $x\in \partial\Omega$. 
Let $\pi^{0}(p')\in \mathbb{R}$ be such that 
$$\min_{p_{N}\in\mathbb{R}} H_{x}(p',p_{N})=H_{x}(p',\pi^{0}(p')).$$
The nonincreasing part of $H_{x}$ is defined by 
$$H_{x}^{-}(p',p_{N})=\left\{ \begin{array}{lll}
H_{x}(p',p_{N}) & \mbox{  if  } &p_{N}\leq \pi^{0}(p') \\
H_{x}(p',\pi^{0}(p')) & \mbox{  otherwise.}
\end{array}\right.$$

\begin{Rmk}
We can show that $H_{x}$ depends continuously on $x\in\partial\Omega$ using the fact that the boundary is of class $\mathcal{C}^{1}.$ To simplify the notations, in the following, we will use $H$ and $H^{-}$ instead of $H_{x}$ and $H_{x}^{-}$ but we will remember that $H^{-}$ and $H$ depends on $x\in\partial\Omega$.
\end{Rmk}

Let us state our main result.

\begin{theo}[Reformulation of state constraints] 
\label{mainth}
~\\
Assume $H:\mathbb{R}^{d+1}\rightarrow\mathbb{R}$ satisfies (\ref{hyp}), and $u:(0,T)\times\bar{\Omega}\rightarrow\mathbb{R}$
then $u$ is a viscosity solution of 
\begin{equation}
\label{sens1}
\left\{ \begin{array}{lll}
u_{t}+H(\nabla u)\geq 0 & \mbox { in } & (0,T)\times\bar{\Omega}\\
u_{t}+H(\nabla u)\leq 0 & \mbox { in } & (0,T)\times\Omega
\end{array}\right.
\end{equation} 
if and only if $u$ is a viscosity solution of the flux-limited problem

\begin{equation}
\label{sens2}
\left\{ \begin{array}{lll}
u_{t}+H(\nabla u) = 0 & \mbox { in } & (0,T)\times\Omega\\
u_{t}+H^{-}(\nabla u)= 0 & \mbox { on } & (0,T)\times\partial\Omega,
\end{array}\right.
\end{equation} 
where $H^{-}$ is the decreasing part of the Hamiltonian along the inward normal vector. 
\end{theo}

Using a similar proof we can show the following stationary counterpart of the previous theorem. 

\begin{theo}[Reformulation of state constraints, stationary case] 
~\\
Assume $H:\mathbb{R}^{d+1}\rightarrow\mathbb{R}$ satisfies (\ref{hyp}) and $u:\bar{\Omega}\rightarrow\mathbb{R}$
then $u$ is a viscosity solution of 
\begin{equation}
\label{sens3}
\left\{ \begin{array}{lll}
u+H(\nabla u)\geq 0 & \mbox { in } & \bar{\Omega}\\
u+H(\nabla u)\leq 0 & \mbox { in } & \Omega
\end{array}\right.
\end{equation} 
if and only if $u$ is a viscosity solution of the flux-limited problem

\begin{equation}
\label{sens4}
\left\{ \begin{array}{lll}
u+H(\nabla u) = 0 & \mbox { in } & \Omega\\
u+H^{-}(\nabla u)= 0 & \mbox { on } & \partial\Omega.
\end{array}\right.
\end{equation} 
\end{theo}

\section{Definition of flux-limited solutions}

We call flux-limited solutions, the viscosity solutions for the flux-limited problem. 
Let us recall the definition of flux-limited solutions in \cite{im} used in our proofs.
The class of test functions we use, is  $\mathcal{C}^{1}((0,T)\times\bar{\Omega}).$ 
We define so-called flux-limiter functions $A: \mathbb{R}^{d}\rightarrow\mathbb{R}$ which are continuous and quasi-convex functions.
We define $F_{A}$ which depends on $x\in\partial\Omega$ as $H^{-}$ depends on $x$, by 
$$F_{A}(p)=\max(A(p'),H^{-}(p',p_{N}))$$ and let 
$$A_{0}(p')=\min_{p_{N}\in \mathbb{R}} H(p',p_{N}).$$
We recall the definition of upper and lower semi-continuous envelopes $u^{*}$ and $u_{*}$ of a (locally bounded) function $u$ defined on $[0,T)\times\bar{\Omega}$,
$$ u^{*}(t,x)=\limsup_{(s,y)\rightarrow (t,x)} u(s,y) \quad \mbox{and} \quad u_{*}(t,x)=\liminf_{(s,y)\rightarrow (t,x)} u(s,y).$$

\begin{defin}[Flux-limited solutions]
\label{solvisc}
Let $H$ satisfies (\ref{hyp}) and $u:(0,T)\times\bar{\Omega} \rightarrow \mathbb{R}$ be locally bounded.
We say that $u$ is a \emph{viscosity sub-solution} (resp. \emph{viscosity super-solution}) in $(0,T)\times\bar{\Omega}$ of
\begin{equation}
\label{eqHJ}
\left\{ \begin{array}{ll}
u_{t}+H(\nabla u) =0 & \mbox{ in } (0,T)\times\Omega\\
u_{t}+F_{A}(\nabla u) =0 & \mbox{ on } (0,T)\times\partial\Omega,
\end{array} \right. 
\end{equation}
if for all $\varphi \in \mathcal{C}^{1}((0,T)\times\bar{\Omega})$ such that 
$u^{*}\leq\varphi$ (resp. $u_{*}\geq\varphi$) in a neighborhood of  $(t_{0,}x_{0}) \in (0,T)\times\bar{\Omega}$
with equality at $(t_{0,}x_{0})$ (we say that $\varphi$ touch $u^{*}$ from above (resp. $u_{*}$ from below) in $(t_{0},x_{0})$ ), we have
$$\left\{ \begin{array}{ll}
\varphi_{t}+H(\nabla \varphi) \leq 0 \quad\mbox{ (resp. } \geq 0 \mbox{ ) } & \mbox{ if } (t_{0},x_{0}) \in (0,T)\times\Omega\\
\varphi_{t}+F_{A}(\nabla \varphi) \leq 0 \quad\mbox{ (resp. } \geq 0 \mbox{ ) } & \mbox{ if } (t_{0},x_{0}) \in (0,T)\times\partial\Omega,
\end{array} \right.$$
and $u^{*}$ 
satisfy 
\begin{equation}
\label{lim2}
\forall (t,x)\in (0,T)\times\partial\Omega \quad u^{*}(t,x)=\limsup_{y\rightarrow x,s\rightarrow t, y\notin \partial\Omega} u^{*}(s,y). 
\end{equation}
Moreover, we say that $u$ is a solution of (\ref{eqHJ}) if $u$ is both a sub and a super-solution.  
\end{defin}

\begin{Rmk} 
We define stationary solutions by replacing $u_{t}$ and $\phi_{t}$ by $u$ and $\phi$ and by giving up time.  
\end{Rmk}

\section{Reduction of test function for a $\mathcal{C}^{1}$ domain}

We give the results and proofs only for the evolution type problem as it is exactly the same for the stationary one where we only have to get rid of time and replace each $u_{t}$ and $\phi_{t}$ by $u$ and $\phi$. 

We define $\pi^{+},\pi^{-}:\mathbb{R}^{d+1}\rightarrow \mathbb{R}$ as in \cite{im2} by 
$$\pi^{+}(p',\lambda)=\inf\{p \in \mathbb{R}: H(p',p)=H^{+}(p',p)=\lambda\},$$
and
$$\pi^{-}(p',\lambda)=\sup\{p \in \mathbb{R}: H(p',p)=H^{-}(p',p)=\lambda\}.$$

Let $X_{0}\in\partial\Omega$ be fixed. For all $X\in\bar{\Omega}$, $X=(X',x^{N})$ where $x^{N}$ is the component on $N_{X_{0}}$ and $X'$ the tangential component. 

Following closely \cite{im,im2}, let us give an equivalent definition of viscosity solutions which reduces the class of test functions. In \cite{im2}, the authors prove the result for hyperplane domain. Here, we prove it for a $\mathcal{C}^{1}$ bounded domain.  

\begin{defin}[Equivalent definition of viscosity solutions]
Let $H$ satisfies (\ref{hyp}) and  $A:\mathbb{R}^{d}\rightarrow\mathbb{R}$ a continuous flux-limiter function such that
$$ \forall p'\in \mathbb{R}^{d}, \quad A(p') \geq A_{0}(p').$$
Let $u:(0,T)\times\bar{\Omega}\rightarrow\mathbb{R}$ be locally bounded. 
We say that $u$ is a \emph{reduced sub-solution} (resp. \emph{reduced super-solution}) of (\ref{eqHJ}) where $F=F_{A}$ in $(0,T)\times\bar{\Omega}$ if and only if $u^{*}$ (resp. $u_{*}$) satisfies (\ref{lim2})
and $u$ is a sub-solution (resp. super-solution) in $(0,T)\times\Omega$ and $\forall \varphi \in \mathcal{C}^{1}((0,T)\times\bar{\Omega})$ such that $u^{*}\leq\varphi$ (resp. $u_{*}\geq\varphi$) in a neighborhood of $(t_{0},X_{0})$ in $(0,T)\times\partial\Omega$ with equality at $(t_{0},X_{0})$, of the following form $$\varphi(t,X',x^{N})=\phi(t,X')+\phi_{0}(x^{N}),$$
where 
$$ \left\{ \begin{array}{l}
\phi \in \mathcal{C}^{1}((0,T)\times\mathbb{R}^{d})\\
D'\phi(t_{0},X_{0}')=p_{0}',
\end{array}\right.$$
and 
$$ \left\{ \begin{array}{l}
\phi_{0} \in \mathcal{C}^{1}(\mathbb{R})\\
\phi_{0}'(x_{0}^{N})=\pi^{+}(p_{0}',A(p_{0}')),
\end{array}\right.$$

we have  $$\varphi_{t} + F_{A}(\nabla\varphi)\leq 0\quad \mbox{ (resp. } \geq 0 \mbox{ ) at } (t_{0},X_{0}) .$$
\end{defin}

\begin{prop}
\label{defequiv}
a) The two previous definitions are equivalent, i.e. a sub-solution (resp. super-solution) of (\ref{eqHJ}) is a reduced one if and only if it is a viscosity sub-solution (resp. super-solution). 
~\\
b) If $u$ is a sub-solution of $u_{t}+H(\nabla u)\leq 0$ in $ (0,T)\times\Omega$ then it is a $A_{0}$-flux limited sub-solution.  
\end{prop}

To show this proposition, let us first recall the lemma taken from  \cite{im,im2} of the critical slope for sub and super-solutions. 
To this purpose, let us define a local parametrization of the smooth domain.  

Let us define the function $\psi^{N}$ on a neighborhood of a fixed point $X_{0} \in \partial\Omega$, 
which gives the component along the inward normal vector $N_{X_{0}}$ of the projection on the boundary of a point in $\bar{\Omega}$ (see \ref{im8}).

\begin{defin}[Function $\psi^{N}$]
\label{defpsi}
Let $X_{0}$ be on $\partial\Omega$. As $\Omega$ is smooth with a boundary of class $\mathcal{C}^{1}$, it exists $r_{0}>0$, $\omega$ an open set of $T_{X_{0}}$ the tangent plane at $X_{0}$ and $\psi^{N} \in \mathcal{C}^{1}(\omega)$  such that $$\partial\Omega\cap\bar{B}_{r_{0}}(X_{0})=\left\{(X',x^{N})\in \omega\times\mathbb{R}, \quad x^{N}=\psi^{N}(X')\right\}.$$
\end{defin}  

\begin{figure}
\begin{center}
  \includegraphics[width=9.0cm]{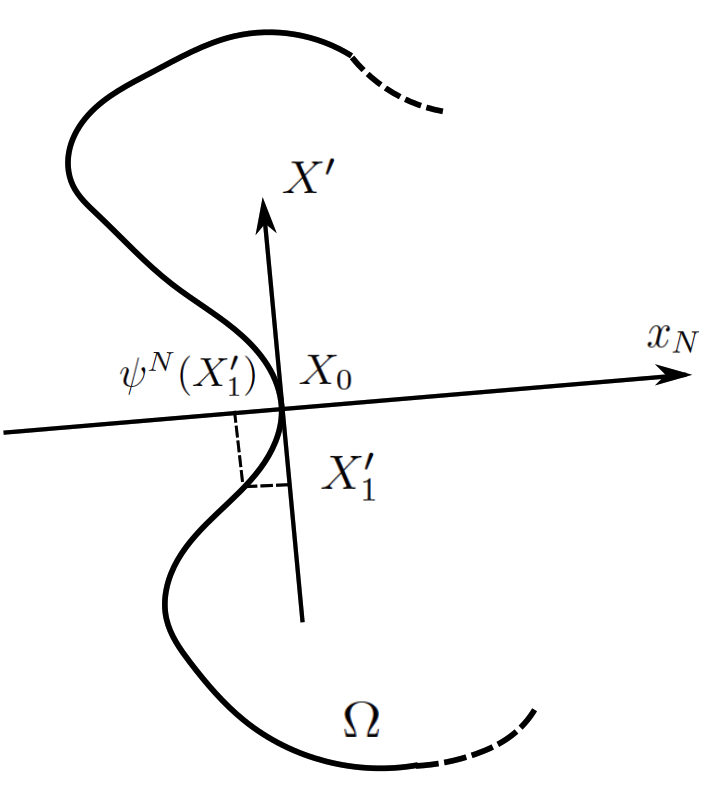}
  \caption{Function $\psi^{N}$ in Definition \ref{defpsi}}
  \label{im8}
  \end{center}
  \end{figure} 

\begin{Rmk}We have
$\nabla'\psi^{N}(X'_{0})=0.$
\end{Rmk}

\begin{Lm}[Critical slope for sub-solution]
\label{pentescri}
Let $u$ be an upper semi-continuous sub-solution of $u_{t}+H(\nabla u) =0$ in $(0,T)\times\Omega$ which satisfies (\ref{lim2})
 and let $\varphi$ be a test function touching $u$ from above at some point $(t_{0},X_{0})$ where $t_{0}\in (0,T)$ and $X_{0}=(X_{0}',x_{0}^{N}) \in \partial\Omega.$ Let $r_{0}$ be given by the definition \ref{defpsi}. Then the critical slope given by 
$$\bar{p}= \inf\left\{ \begin{array}{c}
p\in \mathbb{R}: \exists r \in (0,r_{0}], \quad \varphi(t,X)+p(x^{N}-\psi^{N}(X')) \geq u(t,X),  \\
\forall (t,X)\in (t_{0}-r,t_{0}+r)\times \bar{\Omega}\cap \bar{B}_{r}(X_{0}) 
\end{array}\right\}$$ 
is finite, satisfies $\bar{p} \leq 0$ and 
$$\varphi_{t}(t_{0},X_{0})+H(\nabla'\varphi(t_{0},X_{0}),\partial_{N}\varphi(t_{0},X_{0})+\bar{p})\leq 0.$$
\end{Lm} 

\begin{proof}
Following the same arguments than C. Imbert and R. Monneau in \cite{im}, let us show that $\bar{p}$ is finite in this case. 

Assume $\bar{p}=-\infty$. This implies that there exists $p_{n}\rightarrow-\infty$ and $r_{n}>0$ such that for all $(t,X)\in D_{n}=(t_{0}-r_{n},t_{0}+r_{n})\times\bar{\Omega}\cap \bar{B}_{r_{n}}(X_{0}),$
$$u(t,X)\leq \varphi(t,X)+p_{n}(x^{N}-\psi^{N}(X')).$$
Remark that, replacing $\varphi$ with $\varphi+(t-t_{0})^{2}+||X-X_{0}||^{2}$ if necessary, we can assume that if $(t,X)\neq(t_{0},X_{0}),$ then $u(t,X)<\varphi(t,X)+p_{n}(x^{N}-\psi^{N}(X'))$. In particular, there exists $\delta_{n}>0$ such that on $(t_{0}-r_{n},t_{0}+r_{n})\times \partial B_{r_{n}}(X_{0})\cap\bar{\Omega},$ we have
\begin{equation}
\label{imp}
u(t,X)+\delta_{n}\leq \varphi(t,X)+p_{n}(x^{N}-\psi^{N}(X')).
\end{equation}

Since $u$ satifies (\ref{lim2}), there exists $(t_{\epsilon},X_{\epsilon})\rightarrow (t_{0},X_{0})$ such that $X_{\epsilon}\notin \partial \Omega$ and $u(t_{0},X_{0})=\lim_{\epsilon\rightarrow 0} u(t_{\epsilon},X_{\epsilon}).$ 
Let us introduce now the following perturbed test function
$$\Psi(t,X)=\varphi(t,X)+p_{n}(x^{N}-\psi^{N}(X'))+\epsilon\frac{x_{\epsilon}^{N}-\psi^{N}(X'_{\epsilon})}{x^{N}-\psi^{N}(X')}.$$
Let $(s_{\epsilon},Y_{\epsilon})$ realizing the infimum of $\Psi-u$ in $\bar{D}_{n}$. Using the definition of $\Psi$, we see that $Y_{\epsilon}\notin \partial \Omega$. In particular,

\begin{equation}
\label{ineq}
\varphi(s_{\epsilon},Y_{\epsilon})+p_{n}(y^{N}_{\epsilon}-\psi^{N}(Y'_{\epsilon}))-u(s_{\epsilon},Y_{\epsilon})\leq \Psi(s_{\epsilon},Y_{\epsilon})-u(s_{\epsilon},Y_{\epsilon})\leq \Psi(t_{\epsilon},X_{\epsilon})-u(t_{\epsilon},X_{\epsilon}). 
\end{equation}

The last term of the inequality (\ref{ineq})  goes to $0$ when $\epsilon$ goes to $0$ and using (\ref{imp}), this implies that $(s_{\epsilon},Y_{\epsilon}) \rightarrow (t_{0},X_{0}).$  Since $\Psi$ is a test function at $(s_{\epsilon},Y_{\epsilon})$, we have 

\begin{multline*}
 \varphi_{t}(s_{\epsilon},Y_{\epsilon})+H\left( \nabla\varphi(s_{\epsilon},Y_{\epsilon}) + \left(\begin{array}{c}
 -p_{n}\nabla'\psi^{N}(Y'_{\epsilon})  \\ 
p_{n}
 \end{array} \right)\right.  \\
 \left. -\epsilon\frac{x_{\epsilon}^{N}-\psi^{N}(X'_{\epsilon})}{(y_{\epsilon}^{N}-\psi^{N}(Y'_{\epsilon}))^{2}}\left(\begin{array}{c}
 -\nabla'\psi^{N}(Y'_{\epsilon})  \\ 
1 \end{array} \right) \right) \leq 0.
\end{multline*} 
  
  We can pass to the limit in the viscosity inequality and get 
  $$ \varphi_{t}(t_{0},X_{0})+H\left(  \nabla\varphi(t_{0},X_{0}) +\left(\begin{array}{c}
 0  \\ 
p_{n}
 \end{array} \right)-\alpha\left(\begin{array}{c}
 0  \\ 
1
 \end{array} \right) \right) \leq 0,
  $$
  where $$\alpha=\limsup_{\epsilon\rightarrow 0} \quad \epsilon\frac{x_{\epsilon}^{N}-\psi^{N}(X'_{\epsilon})}{(y_{\epsilon}^{N}-\psi^{N}(Y'_{\epsilon}))^{2}} \in [0,+\infty].$$
  
The previous inequality and the fact that $H$ is coercive implies that $\alpha$ is finite, so that $p_{n}$ cannot go to $-\infty$ as $\varphi_{t}(t_{0},X_{0})$ and $\nabla\varphi(t_{0},X_{0})$ are fixed which gives a contradiction. So $\bar{p}$ is finite.

~\\
Now let us prove the second part of the lemma.
By definition of $\bar{p}$, for all $\epsilon > 0$ small enough, it exists $\delta=\delta(\epsilon)\in(0,\epsilon)$ such that for all $(t,X)\in (t_{0}-\delta,t_{0}+\delta)\times\bar{B}_{\delta}(X_{0})\cap\bar{\Omega},$
$$u(t,X)\leq \varphi(t,X)+(\bar{p}+\epsilon)(x^{N}-\psi^{N}(X')),$$
and it exists $(t_{\epsilon},X_{\epsilon}) \in (t_{0}-\frac{\delta}{2},t_{0}+\frac{\delta}{2})\times B_{\frac{\delta}{2}}(X_{0})$ such that $X_{\epsilon}\notin \partial \Omega$ and
$$u(t_{\epsilon},X_{\epsilon})>\varphi(t_{\epsilon},X_{\epsilon})+(\bar{p}-\epsilon)(x_{\epsilon}^{N}-\psi^{N}(X_{\epsilon}')).$$
Let $\Psi:\mathbb{R}^{d+1}\rightarrow [0,1]$ be a function of class $\mathcal{C}^{\infty}$ such that
$$\Psi=\left\{ \begin{array}{ll}
0 & \mbox{ in } B_{\frac{1}{2}}(0)\\
1 & \mbox{ outside } B_{1}(0). 
\end{array}\right.$$ 
 We define $$\Phi(t,X)=\varphi(t,X)+4\epsilon\delta\Psi\left(\frac{X-X_{0}}{\delta}\right)+(\bar{p}-\epsilon)(x^{N}-\psi^{N}(X')).$$
 So we have on $(t_{0}-\delta,t_{0}+\delta)\times\partial(B_{\delta}(X_{0})\cap\bar{\Omega}),$
 $$ \Phi(t,X) =\left\{ \begin{array}{ll} 
 \varphi(t,X)+4\epsilon\delta+(\bar{p}-\epsilon)(x^{N}-\psi^{N}(X')) & \mbox{ if } X\in \partial B_{\delta}(X_{0}) \\
 \varphi(t,X)+4\epsilon\delta \Psi\left(\frac{X-X_{0}}{\delta}\right) &\mbox{ if } X\in \partial \Omega,
  \end{array}\right.$$
so in all cases, on $(t_{0}-\delta,t_{0}+\delta)\times\partial(B_{\delta}(X_{0})\cap\bar{\Omega}),$ we have
  $$\Phi(t,X)\geq u(t,X).$$
 We also have 
 $$\Phi(t_{0},X_{0})=\varphi(t_{0},X_{0})=u(t_{0},X_{0}),$$
 and 
 $$\Phi(t_{\epsilon},X_{\epsilon})=\varphi(t_{\epsilon},X_{\epsilon})+(\bar{p}-\epsilon)(x_{\epsilon}^{N}-\psi^{N}(X_{\epsilon}')) < u(t_{\epsilon},X_{\epsilon}).$$
 So $u-\Phi$ reaches his maximum in the interior of $(t_{0}-\delta,t_{0}+\delta)\times B_{\delta}(X_{0})\cap\bar{\Omega}$ at some point $(\bar{t}_{\epsilon},\bar{X_{\epsilon}}),$ so $\Phi$ is a test function and 
 $$\Phi_{t}(\bar{t}_{\epsilon},\bar{X_{\epsilon}})+H(\nabla\Phi(\bar{t}_{\epsilon},\bar{X_{\epsilon}})) \leq 0,$$
which gives
 $$\varphi_{t}(\bar{t}_{\epsilon},\bar{X_{\epsilon}})+H\left(\nabla\varphi(\bar{t}_{\epsilon},\bar{X_{\epsilon}})+4\epsilon\nabla\Psi\left(\frac{X_{\epsilon}-X_{0}}{\delta}\right)+\left(\begin{array}{c}
 -(\bar{p}-\epsilon)\nabla'\psi^{N}(X'_{\epsilon}) \\ 
\bar{p}-\epsilon
 \end{array} \right) \right) \leq 0.$$
 When $\epsilon$ goes to $0$, we deduce, as $\nabla'\psi^{N}(X'_{0})=0,$
 $$\varphi_{t}(t_{0},X_{0})+H(\nabla'\varphi(t_{0},X_{0}),\partial_{N}\varphi(t_{0},X_{0})+\bar{p}) \leq 0.$$ 
 \end{proof}

For the super-solutions, we have a similar lemma with a similar proof 
where the fact that the critical slope is finite, is now an hypothesis and not a conclusion.

\begin{Lm}[Critical slope for super-solution]
\label{pentescri2}
Let $u$ be a lower semi-continuous sub-solution of $u_{t}+H(\nabla u) =0$ in $ (0,T)\times\Omega$ 
 and let $\varphi$ be a test function touching $u$ from below at some point $(t_{0},X_{0})$ where $t_{0}\in (0,T)$ and $X_{0}=(X_{0}',x_{0}^{N}) \in \partial\Omega.$ Let $r_{0}$ be given by the definition \ref{defpsi}. If the the critical slope given by 
$$\bar{p}= \sup\left\{ \begin{array}{c}
p\in \mathbb{R}: \exists r \in (0,r_{0}], \quad \varphi(t,X)+p(x^{N}-\psi^{N}(X')) \leq u(t,X),  \\
\forall (t,X)\in (t_{0}-r,t_{0}+r)\times \bar{\Omega}\cap \bar{B}_{r}(X_{0}) 
\end{array}\right\}$$
is finite, then it satisfies $\bar{p} \geq 0$ and we have 
$$\varphi_{t}(t_{0},X_{0})+H(\nabla'\varphi(t_{0},X_{0}),\partial_{N}\varphi(t_{0},X_{0})+\bar{p})\geq 0.$$
\end{Lm} 

Now let us prove \ref{defequiv}.

\begin{proof}[Proof of the proposition \ref{defequiv}]
We first prove the results concerning sub-solutions and then turn to super-solution.
\vskip 0.5cm
\textbf{Sub-solution.}
Let $\phi$ be a test function touching $u^{*}$ from above at $(t_{0},X_{0})\in (0,T)\times\partial\Omega$ and let $\lambda=-\phi_{t}(t_{0},X_{0})$. Let $p'=\nabla'\phi(t_{0},X_{0})$ and $p_{N}=\partial_{N}\phi(t_{0},X_{0}).$ We want to show that 
\begin{equation}
\label{eqsous}
F_{A}(p',p_{N})\leq \lambda.
\end{equation} Notice that by lemma \ref{pentescri}, it exists $\bar{p}\leq 0$ such that 
$$H(p',p_{N}+\bar{p})\leq \lambda,$$ so $A_{0}(p')\leq \lambda$ and as $H^{-}$ is non-increasing, we have
$$\begin{array}{lll}
F_{A}(p',p_{N}) & = & \max(A(p'),H^{-}(p',p_{N}))\\
 & \leq & \max(A(p'),H^{-}(p',p_{N}+\bar{p}))\\
 & \leq&  \max(A(p'),H(p',p_{N}+\bar{p}))\\
 & \leq&  \max(A(p'),\lambda).
\end{array}$$
If $A=A_{0}$, then  $F_{A}(p',p_{N})\leq \lambda$ so b) is proven. 
Assume now that $F_{A}(p',p_{N}) > \lambda$, by contradiction, then we have $A_{0}(p')\leq \lambda < A(p')$
and as $$H(p',p_{N}+\bar{p})\leq \lambda<A(p'),$$ then 
$$p_{N}+\bar{p} < \pi^{+}(p',A(p')).$$
Let us consider the modified test function
$$\varphi(t,X',x^{N})=\phi(t,X',\psi^{N}(X'))-\phi_{0}(\psi^{N}(X'))+\phi_{0}(x^{N}),$$
with $\phi'_{0}(x_{0}^{N})=\pi^{+}(p',A(p')).$
We have 
$$\varphi(t_{0},X_{0})=\phi(t_{0},X_{0})=u^{*}(t_{0},X_{0}) .$$
Let us show that
\begin{equation}
\label{exp10}
\varphi(t,X',x^{N})\geq u^{*}(t,X',x^{N}), 
\end{equation}
on a neighborhood of $(t_{0},X_{0}).$
We have
$$\partial_{N}\phi(t_{0},X_{0})+\bar{p}<\phi'_{0}(x_{0}^{N}),$$
so it exists $p_{1}$ and $p_{2}$ such that $\bar{p}<p_{1}<p_{2}$ and which satisfy
$$\partial_{N}\phi(t_{0},X_{0})+p_{i}<\phi'_{0}(x_{0}^{N}), \quad \forall i \in \{1,2\}.$$
As $\partial_{N}\phi$ and $\phi'_{0}$ are continuous, on a neighborhood of $(t_{0},X_{0}),$ we have
$$\partial_{N}\phi(t,X)+p_{i}<\phi'_{0}(x^{N}), \quad \forall i \in \{1,2\}.$$
So we have
 $$\begin{array}{lll}
\phi(t,X',x^{N}) &=& \phi(t,X',\psi^{N}(X'))+\displaystyle\int_{\psi^{N}(X')}^{x^{N}}\partial_{N}\phi(t,X',s)\mathrm{d}s \\
 & = & \varphi(t,X',x^{N})+\phi_{0}(\psi^{N}(X'))-\phi_{0}(x^{N})+\displaystyle\int_{\psi^{N}(X')}^{x^{N}}\partial_{N}\phi(t,X',s)\mathrm{d}s\\
  & = & \varphi(t,X',x^{N})+\displaystyle\int_{\psi^{N}(X')}^{x^{N}}(\partial_{N}\phi(t,X',s)-\phi'_{0}(s))\mathrm{d}s \\
  & \leq & \varphi(t,X',x^{N}) -p_{2}(x^{N}-\psi^{N}(X')),
\end{array}$$
and by definition of $\bar{p},$ it exists a neighborhood $(t_{0}-r,t_{0}+r)\times\bar{\Omega}\cap B_{r}(X_{0})$ of $(t_{0},X_{0}),$ for some $r>0$ such that 
$$\begin{array}{rll}
u^{*}(t,X',x^{N}) & \leq & \phi(t,X',x^{N})+p_{1}(x^{N}-\psi^{N}(X'))
    \\                   & \leq & \varphi(t,X',x^{N})+(p_{1}-p_{2})(x^{N}-\psi^{N}(X')),
\end{array}$$
so we get (\ref{exp10}).
~\\
By the definition of reduced sub-solution and using that
$$H^{-}(p',\pi^{+}(p',A(p')))=A_{0}(p')\leq A(p'),$$
we deduce that
$$A(p')=F_{A}(p',\pi^{+}(p',A(p')))\leq \lambda.$$
which gives a contradiction. Therefore (\ref{eqsous}) holds true. 
Let us prove now the super-solution case.
\vskip 0.5cm
\textbf{Super-solution.} Let $\phi$ be a test function touching $u_{*}$ from below at $(t_{0},X_{0})\in (0,T)\times\partial\Omega.$ Let $\lambda=-\phi_{t}(t_{0},X_{0}).$
We want to show that 
\begin{equation}
\label{sursol}
F_{A}(\phi_{x})\geq\lambda.
\end{equation}
By the lemma \ref{pentescri2}, if $\bar{p}$ is finite, then $\bar{p} \geq 0$ and
\begin{equation}
\label{exp1}
H(p',p_{N}+\bar{p})\geq\lambda. 
\end{equation}
If $\bar{p}=+\infty$ then as $H$ is coercive, the inequality (\ref{exp1}) is true for some $\tilde{p}$ big enough. 
To simplify the notations,  $\bar{p}$ will denote the real number satisfying the inequality (\ref{exp1}) in the first or the second case. 
Note that (\ref{sursol}) holds true if $$\lambda \leq A(p') \quad \mbox{ or } \quad H^{-}(p',p_{N}+\bar{p})=H(p',p_{N}+\bar{p}).$$
Assume by contradiction that the inequality (\ref{sursol}) is wrong. 
Then we have, 
\begin{equation}
\label{exp2}
A_{0}(p')\leq A(p)<\lambda\leq H^{+}(p',p_{N}+\bar{p}).
\end{equation}
As $H^{-}(p',p_{N}+\bar{p})<H(p',p_{N}+\bar{p}),$ we deduce in particular that
$$p_{N}+\bar{p}>\pi^{+}(p',A(p'))=\phi'_{0}(x_{0}^{N}),$$
because if by contradiction, 
$$p_{N}+\bar{p}\leq\pi^{+}(p',A(p')),$$
holds true, then
$$\lambda\leq H(p',p_{N}+\bar{p})=H^{+}(p',p_{N}+\bar{p})\leq H^{+}(p',\pi^{+}(p',A(p'))=A(p'),$$
which is a contradiction with (\ref{exp2}).
As before, with the same proof that the sub-solution case, we show that the function
$$\varphi(t,X',x^{N})=\phi(t,X',\psi^{N}(X'))-\phi_{0}(\psi^{N}(X'))+\phi_{0}(x^{N}),$$
is a test function touching $u$ from below at $(t_{0},X_{0}).$ By definition of reduced sub-solution,
$$A(p')=F_{A}(p',\pi^{+}(p',A(p')))\geq \lambda,$$
which gives a contradiction with (\ref{exp2}). Therefore (\ref{sursol}) holds true. 
\end{proof}

Now we can prove the reformulation of state constraints theorem.

\section{Proof of the theorem \ref{mainth}}
The proposition \ref{defequiv} about the equivalence of definition allows to reduce the set of test functions. Now let us prove the theorem using the same ideas as in \cite{im}.

\begin{proof}[Proof of the theorem \ref{mainth}]
We do the proof in three steps.
\newline
\textbf{1st step:}
Let us prove that
$$u_{t}+H(\nabla u) \leq 0 \quad \mbox{ in } (0,T)\times\Omega,$$
implies
$$u_{t}+H^{-}(\nabla u) \leq 0\quad \mbox{ on } (0,T)\times\partial\Omega.$$
Using b) of the proposition \ref{defequiv}, $u$ is a $A_{0}$-flux limited sub-solution, so
$$u_{t}+F_{A_{0}}(\nabla u) \leq 0 \quad \mbox{ on } (0,T)\times\partial\Omega.$$ 
As $F_{A_{0}}(\nabla u)=H^{-}(\nabla u)$, we have $$u_{t}+H^{-}(\nabla u)\leq 0\quad \mbox{ on } (0,T)\times\partial\Omega,.$$
\textbf{2nd step:}
Let us prove that
$$u_{t}+H(\nabla u) \geq 0\quad \mbox{ in } (0,T)\times\bar{\Omega},$$
implies
$$u_{t}+H^{-}(\nabla u) \geq 0 \quad\mbox{ on } (0,T)\times\partial\Omega.$$
Let $\varphi$ be a test function touching $u_{*}$ from below at $(t_{0},X_{0})\in (0,T)\times\partial\Omega.$
Using the proposition \ref{defequiv}, we assume that
$$\varphi(t,X)=\phi(t,X')+\phi_{0}(x^{N}),$$
where $$\phi\in \mathcal{C}^{1}((0,T)\times\mathbb{R}^{d}),\quad \nabla'\phi(t_{0},X_{0}')=p'_{0},$$
and $$\phi_{0}\in\mathcal{C}^{1}(\mathbb{R}), \quad \phi'_{0}(x_{0}^{N})=\pi^{+}(p'_{0},A_{0}(p'_{0})).$$
As $$H(\nabla \varphi)=H^{+}(\nabla \varphi)=A_{0}(p'_{0})=H^{-}(\nabla\varphi),$$
and by hypothesis, we have $\varphi_{t}+H(\nabla \varphi)\geq 0$,
we deduce $$\varphi_{t}+H^{-}(\nabla \varphi)\geq 0.$$
\textbf{3rd step:}
The reverse come from the fact that $H^{-}\leq H.$
\end{proof}

\section{Simpler proofs in particular cases}\label{app:app1}

Let us give some direct proof in some cases without using the equivalence of definition which reduces the set of test function.

\subsection{Multi-dimensional case for super-solutions}
Let us give a direct proof of the following theorem which is a part of the general reformulation of state constraints theorem.

\begin{theo}
If $u$ satisfies 
$$u_{t}+H(\nabla u) \geq 0\quad \mbox{ on } (0,T)\times\partial\Omega$$
then 
$$u_{t}+H^{-}(\nabla u) \geq 0\quad  \mbox{ on } (0,T)\times\partial\Omega.$$
\end{theo} 

\begin{proof}
Let  $\varphi$ be a test function touching $u_{*}$ from below at $(t_{0},X_{0})\in (0,T)\times \partial\Omega$ then
$$\varphi_{t}+H(\nabla'\varphi,\partial_{N}\varphi)\geq 0 \quad \mbox{ at } (t_{0},X_{0}),$$ and let $\lambda=-\varphi_{t}(t_{0},X_{0})$,
so $$\partial_{N}\varphi\leq \pi^{-}(\nabla'\varphi,\lambda) \quad \mbox{ or } \quad \partial_{N}\varphi\geq \pi^{+}(\nabla'\varphi,\lambda).$$
If $\partial_{N}\varphi(t_{0},X_{0})\leq \pi^{-}(\nabla'\varphi,\lambda)$ then $$H^{-}(\nabla\varphi)=H(\nabla\varphi)$$ so 
$$\varphi_{t}+H^{-}(\nabla'\varphi,\partial_{N}\varphi)\geq 0.$$
If $\partial_{N}\varphi(t_{0},X_{0})\geq \pi^{+}(\nabla'\varphi,\lambda)$, we consider $$\phi(t,X)=\varphi(t,X)+(\pi^{-}(\nabla'\varphi,A_{0}(\nabla'\varphi))-\partial_{N}\varphi(t_{0},X_{0}))(x^{N}-\psi^{N}(X')),$$
where $\psi^{N}$ is the function defined previously, so $\phi$ is defined on a neighborhood $(0,T)\times\bar{B}_{r_{0}}(X_{0})\cap \bar{\Omega}$ of $(t_{0},X_{0}).$
Using that $$\pi^{-}(\nabla'\varphi,A_{0}(\nabla'\varphi))\leq \pi^{+}(\nabla'\varphi,\lambda)$$ and $$\forall X\in \bar{B}_{r_{0}}(X_{0})\cap \bar{\Omega}, \quad x^{N}\geq \psi^{N}(X'),$$ we have $$\forall (t,X)\in (0,T)\times\bar{B}_{r_{0}}(X_{0})\cap \bar{\Omega}, \quad \phi(t,X)\leq \varphi(t,X),$$
so $\phi$ is a test function such that $$\partial_{N}\phi(t_{0},X_{0})=\pi^{-}(\nabla'\varphi,A_{0}(\nabla'\varphi)),$$ $$\nabla'\phi(t_{0},X_{0})= \nabla'\varphi(t_{0},X_{0})$$ as $\nabla'\psi^{N}(X'_{0})=0$ and $$\phi(t_{0},X_{0})=\varphi(t_{0},X_{0}).$$
So $H(\nabla'\phi, \partial_{N}\phi(t_{0},X_{0}))=A_{0}(\nabla'\phi)$ and as $\phi$ is a test function,
$$\phi_{t}(t_{0},X_{0})+A_{0}(\nabla'\phi)\geq 0,$$
so we have
\begin{equation*}
\varphi_{t}(t_{0},X_{0})+H^{-}(\nabla'\varphi,\partial_{N}\varphi)\geq \varphi_{t}(t_{0},X_{0})+A_{0}(\nabla'\varphi)\geq 0. \qedhere
\end{equation*}
\end{proof}

\subsection{The stationary sub-solution case in dimension 1}

Let us give a proof of the following theorem which is a part of the reformulation of state constraints theorem for stationary sub-solutions in dimension 1.

\begin{theo}
\label{dim1}
Let $\Omega=]a,b[.$
If $u$ satisfies (\ref{lim2}) and
$$u+H(\nabla u) \leq 0 \quad \mbox{ in } \Omega$$
then
$$u+H^{-}(\nabla u) \leq 0 \quad \mbox{ on } \partial\Omega.$$
\end{theo} 

\begin{figure}
\begin{center}
  \includegraphics[width=12.0cm]{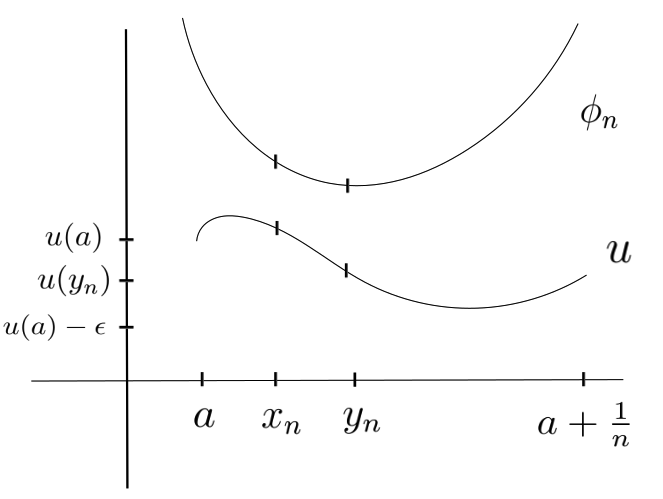}
  \caption{Illustration of the function $\phi_{n}$ in Theorem \ref{dim1}}
  \label{im7}
  \end{center}
  \end{figure}

\begin{proof}
Let $\psi$ be a test function such that $\psi\geq u^{*}$ and $\psi(a)=u^{*}(a).$
To simplify the notations, we will use $u$ instead of $u^{*}$.
Let $\epsilon>0.$ Let us first construct a sequence $(x_{n})_{n\in \mathbb{N}}$ which tends to $a$ such that
$$u(x_{n})\geq u(a)-\epsilon,$$
and such that $u$ allows a test function $\phi_{n}$ at $x_{n}$ such that 
$$\phi_{n}(a)\geq u(a).$$
\newline
\textbf{1st step: Construction of $x_{n}$ and $\phi_{n}$.}~\\
We define $\Omega_{n}=]a,a+\frac{1}{n}[$. As $$u(a)=\limsup_{y\rightarrow a,y\neq a} u(y),$$
it exists $ y_{n}\in \Omega_{n}$ such that $$u(y_{n})\geq u(a)-\epsilon.$$
Let $\varphi_{n}\in C^{1}$ such that $$\min_{\bar{\Omega}_{n}} \varphi_{n}=\varphi(y_{n}),$$
and $$\varphi_{n}(a)-u(a)>\varphi_{n}(y_{n})-u(y_{n})\geq 0,$$
and $$\varphi_{n}(a+\frac{1}{n})-u(a+\frac{1}{n})>\varphi_{n}(y_{n})-u(y_{n})\geq 0.$$
The function $\varphi_{n}$ is drawn in \ref{im7}.
So $$\min_{\bar{\Omega}_{n}} \varphi_{n}-u=\varphi_{n}(x_{n})-u(x_{n}) \quad \mbox{ where }x_{n}\in\Omega_{n}.$$
We define $\phi_{n}=\varphi_{n}+c_{n}$ such that $\phi_{n}(x_{n})=u(x_{n})$. We get the sequence $(x_{n})_{n\in\mathbb{N}}$.
\vspace{0,5cm}
\linebreak
\textbf{2nd step: Conclusion.}
\newline
As $\phi_{n}$ is a test function at $x_{n},$
$$\phi_{n}+H(\phi'_{n})\leq 0\quad \mbox{ at } x_{n},$$
so 
$$H(\phi'_{n})\leq -u(x_{n})\leq -u(a)+\epsilon,$$
so $$\pi^{-}(-u(a)+\epsilon)\leq\phi'_{n}(x_{n}).$$
Multiplying by $x_{n}-a$, we have
$$\begin{array}{lll}
\pi^{-}(-u(a)+\epsilon)(x_{n}-a) & \leq &\phi'_{n}(x_{n})(x_{n}-a) \\
& \leq & \phi_{n}(x_{n})-\phi_{n}(a)+o_{n\rightarrow +\infty}(x_{n}-a)\\
 & \leq & u(x_{n})-u(a)+o(x_{n}-a)\\
 & \leq & \psi(x_{n})-\psi(a)+o(x_{n}-a)\\
 & \leq & \psi'(a)(x_{n}-a)+o(x_{n}-a).
\end{array}$$
We deduce that $$\pi^{-}(-u(a)+\epsilon)\leq \psi'(a),$$
so $$\forall \epsilon >0 \quad u(a) +H^{-}(\psi'(a))\leq \epsilon,$$
Letting $\epsilon$ go to $0$, we have $$\psi(a)+H^{-}(\psi'(a))\leq 0.$$
\end{proof}

\subsection{Critical slope $\bar{p}$ is finite in the stationary case in Lemma \ref{pentescri} }

Let us give a simpler proof of the fact that the critical slope $\bar{p}$ is finite in the stationary case in Lemma \ref{pentescri}.

\begin{Lm}
Let $u$ be an upper semi-continuous sub-solution of $u+H(\nabla u) =0$ in $\Omega$ which satisfies (\ref{lim2})
 and let $\varphi$ be a test function touching $u$ from above at some point $X_{0}$ and $X_{0}=(X_{0}',x_{0}^{N}) \in \partial\Omega.$ Let $r_{0}$ be given by the definition \ref{defpsi}. Then the critical slope given by 
$$\bar{p}= \inf\left\{
p\in \mathbb{R}: \exists r \in (0,r_{0}], \quad \varphi(X)+p(x^{N}-\psi^{N}(X')) \geq u(X),  \quad
\forall X \in  \bar{\Omega}\cap \bar{B}_{r}(X_{0}) 
\right\}$$ 
is finite.
\end{Lm}

\begin{proof}
In the stationary case, let us give another proof of the fact that $\bar{p}$ is finite.
The function $u$ is a sub-solution so is bounded on the compact set $\bar{\Omega}$. Let $M>0$ be such that $|u|\leq M$. 
So for all test function $\phi$ touching from above at $X\in \Omega$, we have
$$ H(\nabla\phi)\leq -u(X)\leq M.$$
As $H$ is coercive, $|\nabla\phi|\leq C$ for some $C>0$, so $u$ is $C$-lipschitz on $\Omega$ by the proposition 1.14 of Ishii in \cite{ishbar}. 
Using the assumption (\ref{lim2}), for $(x,y) \in \bar{\Omega}^{2}$, it exists a sequence $(x_{n})_{n\in\mathbb{N}}$ of  $\Omega$ which tends to $x$ such that $\lim_{n\rightarrow +\infty}u(x_{n})=u(x)$ and a sequence $(y_{n})_{n\in\mathbb{N}}$ of $\Omega$ which tends to $y$ such that $\lim_{n\rightarrow +\infty}u(y_{n})=u(y),$ so 
$$|u(x_{n})-u(y_{n})|\leq C|x_{n}-y_{n}|,$$ and letting $n$ go to $+\infty,$ we have
$$|u(x)-u(y)|\leq C|x-y|,$$
so $u$ is $C$-lipschitz on $\bar{\Omega}.$
Let $\varphi$ be a test function touching $u$ from above at $X_{0}\in \partial\Omega,$ (see definition \ref{solvisc}). 
As the function $\varphi$ is of class $\mathcal{C}^{1},$ it is $C'$-lipschitz on $\bar{\Omega}$, so for $p$ satisfying
$$\varphi(X)-\varphi(X_{0})+p(x^{N}-\psi^{N}(X')) \geq u(X) - u(X_{0}),$$
on $B_{r}(X_{0})$,
we deduce that
$$C'|X-X_{0}|+p(x^{N}-\psi^{N}(X')) \geq -C|X-X_{0}|,$$
so taking $X'=X_{0}'$ et $x^{N}\geq x_{0}^{N},$ we have 
$$C'(x^{N}-x_{0}^{N})+p(x^{N}-x_{0}^{N})\geq -C(x^{N}-x_{0}^{N}),  $$
so $p\geq -(C+C'),$ and $\bar{p}$ is finite.
\end{proof}

\textbf{Acknowledgements.} The author thanks C. Imbert for all his advices concerning this work and R. Monneau for discussions about the motivation of the problem. This work was partially supported by the ANR-12-BS01-0008-01 HJnet project.

\bibliography{bibli}
\bibliographystyle{unsrt}

\end{document}